\documentclass{amse-new}

\numberwithin{equation}{section} 

\begin{document}

 \PageNum{1}
 \Volume{201x}{Sep.}{x}{x}
 \OnlineTime{August 15, 201x}
 \DOI{0000000000000000}
 \EditorNote{Received x x, 201x, accepted x x, 201x}

\abovedisplayskip 6pt plus 2pt minus 2pt \belowdisplayskip 6pt
plus 2pt minus 2pt
\def\vsp{\vspace{1mm}}
\def\th#1{\vspace{1mm}\noindent{\bf #1}\quad}
\def\proof{\vspace{1mm}\noindent{\it Proof}\quad}
\def\no{\nonumber}
\newenvironment{prof}[1][Proof]{\noindent\textit{#1}\quad }
{\hfill $\Box$\vspace{0.7mm}}
\def\q{\quad} \def\qq{\qquad}
\allowdisplaybreaks[4]


\AuthorMark{Dinghuai Wang and Jiang Zhou}                             

\TitleMark{Central BMO spaces with variable exponent}  

\title{Central BMO spaces with variable exponent    
\footnote{Supported by NSFC (Grant No. 11261055 and No. 11271312)\\  $^\ast$Corresponding author}}                  

\author{Dinghuai \uppercase{Wang}$^{1}$, Zongguang Liu$^{2}$, Jiang \uppercase{Zhou}$^{\ast,1}$ and Zhidong Teng$^{1}$}           
    {1.College of Mathematics and System Sciences, Xinjiang University, Urumqi 830046, China\\
    E-mail\,$:$ Wangdh1990@126.com; zhoujiang@xju.edu.cn; liuzg@cumtb.edu.cn\\
2.Department of Mathematics, China University of Mining and Technology, Beijing 100083, China\\
E-mail\,$:$ liuzg@cumtb.edu.cn\\}

\maketitle%

\Abstract{In this paper, the central BMO spaces with variable exponent are introduced. As an application, we characterize these spaces by the boundedness of commutators of Hardy operator and its dual operator on variable Lebesgue spaces. The boundedness of vector-valued commutators on Herz spaces with variable exponent are also considered.}      


\MRSubClass{42B35,47B47.}      

\section{Introduction}

Given a measurable function $p: \mathbb{R}^n\rightarrow [1,\infty)$, $L^{p(\cdot)}(\mathbb{R}^n)$ denotes the set of measurable functions $f$ such that for some $\lambda>0$,
$$\int_{\mathbb{R}^n}\Big(\frac{|f(x)|}{\lambda}\Big)^{p(x)}dx<\infty.$$
$L^{p(\cdot)}(\mathbb{R}^n)$ is a Banach function space when equipped with the norm
$$\|f\|_{L^{p(\cdot)}(\mathbb{R}^n)}=\inf\Big\{\lambda>0:\int_{\mathbb{R}^n}\Big(\frac{|f(x)|}{\lambda}\Big)^{p(x)}dx\leq 1\Big\}.$$
Such space is called a {\it variable Lebesgue space}, since it is generalized from the standard Lebesgue
space. Variable Lebesgue spaces become one of the most important function spaces due to the fundamental paper \cite{KR} by Kov\'{a}\v{c}ik and R\'{a}kosn\'{i}k. Recently, Cruz-Uribe, Fiorenza, Martell and P\'{e}rez \cite{CFMP} proved that many classical operators in harmonic analysis, such as maximal operators, singular integrals, commutators and fractional integrals are bounded on the variable Lebesgue space.

On the other hand, it is well known that ${\rm BMO}(\mathbb{R}^n)$ space is just the dual space of Hardy space. Like this, the dual space of Herz-type Hardy space is central {\rm BMO} space. The so-called central {\rm BMO} space is defined by
$${\rm CBMO}^{p}(\mathbb{R}^n)=\{f\in L_{\rm loc}^{p}(\mathbb{R}^n): \|f\|_{{\rm CBMO}^{p}(\mathbb{R}^n)}<\infty\}$$
with
$$\|f\|_{{\rm CBMO}^{p}(\mathbb{R}^n)}=\sup_{r>0}\bigg(\frac{1}{|B(0,r)|}\int_{B(0,r)}|f(x)-f_{B(0,r)}|^{p}dx\bigg)^{1/p}.$$
The space ${\rm CBMO}^{p}(\mathbb{R}^n)$ can be regarded as a local version of ${\rm BMO}(\mathbb{R}^n)$ at
the origin. In fact, ${\rm BMO}(\mathbb{R}^n)\subsetneqq {\rm CBMO}^{p}(\mathbb{R}^n)$, where $1\leq p < \infty$. But, facts about ${\rm CBMO}^{p}(\mathbb{R}^n)$ are not always
simpler versions of facts about ${\rm BMO}(\mathbb{R}^n)$. For example, there is no John-Nirenberg inequality \cite{JN} for ${\rm CBMO}^{p}(\mathbb{R}^n)$. See also \cite{CL}, \cite{G}, \cite{LY} and\cite{WZ} for more details. In 2007, Fu, Lu, Liu and Wang \cite{FLLW} characterized ${\rm CBMO}^p(\mathbb{R}^n)$ space in terms of the boundedness of commutators of fractional Hardy operator.

In this paper, we will introduce the spaces of central ${\rm BMO}$ with variable exponent and characterize these spaces by the boundedness of commutators of Hardy operator and its dual operator on variable Lebesgue spaces. The boundedness of vector-valued commutators on Herz spaces with variable exponent are also considered.

Throughout this paper, the letter $C$ denotes constants which are independent of main variables and may change from one occurrence to another.
Let $B_{k}=\{x\in \mathbb{R}^n:|x|<2^{k}\}$ and $C_{k}=B_{k}\backslash B_{k-1}$ for $k\in \mathbb{Z}$. Denote $\chi_{k}=\chi_{C_{k}}$.

\section{Properties of variable exponent}

Denote by $\mathcal{P}(\mathbb{R}^n)$ the set of measurable functions $p$ on $\mathbb{R}^n$ such that $p_{-}>1$ and $p_{+}<\infty$, where
$$p_{-}:={\rm ess~inf}\{p(x):x\in\mathbb{R}^n\}, \quad p_{+}:={\rm ess~sup}\{p(x):x\in\mathbb{R}^n\}.$$
$p'(\cdot)$ means the conjugate exponent of $p(\cdot)$, namely $1/p(\cdot)+1/p'(\cdot)=1$ holds.

If $p(\cdot)\in \mathcal{P}(\mathbb{R}^n)$, then the norm $\|f\|_{L^{p(\cdot)}(\mathbb{R}^n)}$ is equivalent to
$$\sup\bigg\{\Big|\int_{\mathbb{R}^n}f(x)g(x)dx\Big|:\|g\|_{L^{p'(\cdot)}(\mathbb{R}^n)}\leq 1\bigg\}.$$
More precisely,
\begin{equation}\label{1.1}
\|f\|_{L^{p(\cdot)}(\mathbb{R}^n)}\leq \sup\bigg\{\Big|\int_{\mathbb{R}^n}f(x)g(x)dx\Big|:\|g\|_{L^{p'(\cdot)}(\mathbb{R}^n)}\leq 1\bigg\}\leq r_{p}\|f\|_{L^{p(\cdot)}(\mathbb{R}^n)},
\end{equation}
where $r_{p}:=1+1/p_{-}+1/p_{+}$ (see \cite{KR}).

Given a function $f\in L^{1}_{\rm loc}(\mathbb{R}^n)$, the Hardy-Littlewood maximal operator $M$ is defined by
$$M(f)(x):=\sup_{r>0}r^{-n}\int_{B(x,r)}|f(y)|dy.$$
The set $\mathcal{B}(\mathbb{R}^n)$ is of $p(\cdot)\in \mathcal{P}(\mathbb{R}^n)$ satisfying the condition that $M$ is bounded on $L^{p(\cdot)}(\mathbb{R}^n)$. It is well known that the boundedness of the Hardy-Littlewood maximal operator on Lebesgue spaces plays a key role in analysis. It also does in variable exponent Lebesgue spaces.

Next, we state some properties of variable exponent belonging to the class $\mathcal{B}(\mathbb{R}^n)$ (see \cite{D}).
\begin{lemma}\label{l1}
Let $p(\cdot)\in \mathcal{P}(\mathbb{R}^n)$. Then the following conditions are equivalent:
\begin{enumerate}
\item [\rm(1)] $p(\cdot)\in \mathcal{B}(\mathbb{R}^n)$;
\item [\rm(2)] $p'(\cdot)\in \mathcal{B}(\mathbb{R}^n)$;
\item [\rm(3)] $p(\cdot)/p_{0}\in \mathcal{B}(\mathbb{R}^n)$ for some $1<p_{0}<p_{-}$;
\item [\rm(4)] $\big(p(\cdot)/p_{0}\big)'\in \mathcal{B}(\mathbb{R}^n)$ for some $1<p_{0}<p_{-}$.
\end{enumerate}
\end{lemma}

Below $\mathcal{Y}$ denotes all families of disjoint and open cubes in $\mathbb{R}^n$. The next Lemma is due to Diening \cite[Lemma 5.5]{D}. We remark that Diening has proved general results on Musielak-Orlicz spaces. We describe them for Lebesgue spaces with variable exponent.

\begin{lemma}\label{l2}
If $p(\cdot)\in \mathcal{B}(\mathbb{R}^n)$, then there exist constants $C>0$ and $0<\delta<1$ such that for all $Y\in \mathcal{Y}$, all non-negative numbers $t_{Q}$ and all $f\in L^{1}_{\rm loc}(\mathbb{R}^n)$ with $f_{Q}\neq 0 (Q\in \mathcal{Y})$,
\begin{equation}\label{eq2.1}
\bigg\|\sum_{Q\in \mathcal{Y}}t_{Q}\Big|\frac{f}{f_{Q}}\Big|^{\delta}\chi_{Q}\bigg\|_{L^{p(\cdot)}(\mathbb{R}^n)}\leq C\bigg\|\sum_{Q\in \mathcal{Y}}t_{Q}\chi_{Q}\bigg\|_{L^{p(\cdot)}(\mathbb{R}^n)}.
\end{equation}
\end{lemma}
\vspace{0.3cm}
The following results are the key lemmas due to Izuki \cite{I}.
\begin{lemma}\label{l3}
If $p(\cdot)\in \mathcal{B}(\mathbb{R}^n)$, then there exist constants $C>0$ such that for all balls $B$ in $\mathbb{R}^n$,
\begin{equation}\label{eq2.2}
\frac{1}{|B|}\|\chi_{B}\|_{L^{p(\cdot)}(\mathbb{R}^n)}\|\chi_{B}\|_{L^{p'(\cdot)}(\mathbb{R}^n)}\leq C.
\end{equation}
\end{lemma}

\begin{lemma}\label{l4}
If $p(\cdot)\in \mathcal{B}(\mathbb{R}^n)$, then there exist constants $\delta, C>0$ such that for all balls in $\mathbb{R}^n$ and all measurable subsets $S\subset B$,
\begin{equation}\label{eq2.3}
\frac{\|\chi_{B}\|_{L^{p(\cdot)}(\mathbb{R}^n)}}{\|\chi_{S}\|_{L^{p(\cdot)}(\mathbb{R}^n)}}\leq C\frac{|B|}{|S|},
\end{equation}
\begin{equation}\label{eq2.4}
\frac{\|\chi_{S}\|_{L^{p(\cdot)}(\mathbb{R}^n)}}{\|\chi_{B}\|_{L^{p(\cdot)}(\mathbb{R}^n)}}\leq C\bigg(\frac{|S|}{|B|}\bigg)^{\delta}.
\end{equation}
\end{lemma}

In fact, the result of \eqref{eq2.3} in Lemma \ref{l4} can be improved as follows.
\begin{lemma}\label{l5}
If $p(\cdot)\in \mathcal{B}(\mathbb{R}^n)$ and $1<p_{0}<p_{-}$, then there exists a constant $C>0$ such that for all balls $B$ in $\mathbb{R}^n$ and all measurable subsets $S\subset B$,
\begin{equation}\label{eq2.5}
\frac{\|\chi_{B}\|_{L^{p(\cdot)}(\mathbb{R}^n)}}{\|\chi_{S}\|_{L^{p(\cdot)}(\mathbb{R}^n)}}\leq C\bigg(\frac{|B|}{|S|}\bigg)^{1/p_{0}},
\end{equation}
\end{lemma}
\begin{proof}
Take a ball $B$ and a measurable subset $S\subset B$ arbitrarily. By Lemma \ref{l1}, we obtain that $p(\cdot)/p_{0}\in \mathcal{B}(\mathbb{R}^n)$, then
$$\frac{\|\chi_{B}\|_{L^{p(\cdot)/p_{0}}(\mathbb{R}^n)}}{\|\chi_{S}\|_{L^{p(\cdot)/p_{0}}(\mathbb{R}^n)}}\leq C\frac{|B|}{|S|}.$$
From the fact that $\|f\|_{L^{p(\cdot)}(\mathbb{R}^n)}=\||f|^{p_{0}}\|^{1/p_{0}}_{L^{p(\cdot)/p_{0}}(\mathbb{R}^n)}$ and $(\chi_{B})^{p_{0}}=\chi_{B}$, we compute that
$$\frac{\|\chi_{B}\|_{L^{p(\cdot)}(\mathbb{R}^n)}}{\|\chi_{S}\|_{L^{p(\cdot)}(\mathbb{R}^n)}}
=\bigg(\frac{\|\chi_{B}\|_{L^{p(\cdot)/p_{0}}(\mathbb{R}^n)}}{\|\chi_{S}\|_{L^{p(\cdot)/p_{0}}(\mathbb{R}^n)}}\bigg)^{1/p_{0}}\leq C\bigg(\frac{|B|}{|S|}\bigg)^{1/p_{0}}.$$
Therefore we have proved Lemma \ref{l5}.
\end{proof}

\section{Central BMO spaces with variable exponent}

In this section, we introduce the definition of central BMO spaces with variable exponent ${\rm CBMO}^{p(\cdot)}$ and study the properties of ${\rm CBMO}^{p(\cdot)}$.

Denote $$L^{p(\cdot)}_{\rm loc}(\mathbb{R}^n):=\Big\{f: f\in L^{p(\cdot)}(K) ~\text{for all compact subsets}~ K\subset \mathbb{R}^n\Big\}$$
and let $p(\cdot)\in \mathcal{P}(\mathbb{R}^n)$. A function $f\in L^{p(\cdot)}_{{\rm loc}}(\mathbb{R}^n)$ is said to belong to central {\rm BMO} space with variable exponent, if
$$\|f\|_{{\rm CBMO}^{p(\cdot)}(\mathbb{R}^n)}
:=\sup_{r>0}\|\chi_{B(0,r)}\|^{-1}_{L^{p(\cdot)}(\mathbb{R}^n)}\|(f-f_{B(0,r)})\chi_{B(0,r)}\|_{L^{p(\cdot)}(\mathbb{R}^n)}<\infty.$$
If $p(x)=p$ is constant, then ${\rm CBMO}^{p(\cdot)}(\mathbb{R}^n)$ equals ${\rm CBMO}^{p}(\mathbb{R}^n)$. Here and in what follows, we write ${\rm C}^{p(\cdot)}:={\rm CBMO}^{p(\cdot)}(\mathbb{R}^n)$ simple.
\vskip 0.2cm

Now, we show the relationship between ${\rm C}^{p(\cdot)}$ and central {\rm BMO} spaces.
\begin{proposition}\label{p1}
If $p(\cdot)\in \mathcal{B}(\mathbb{R}^n)$, then ${\rm C}^{p(\cdot)}\subsetneq {\rm CBMO}(\mathbb{R}^n)$.
\end{proposition}
\begin{proof}
For any $B:=B(0,r)$, applying \eqref{1.1} and Lemmas \ref{l3}, we have
\begin{eqnarray*}
\frac{1}{|B|}\int_{B}|f(x)-f_{B}|dx&\leq&\frac{C}{|B|}\|(f-f_{B})\chi_{B}\|_{L^{p(\cdot)}}\|\chi_{B}\|_{L^{p'(\cdot)}}\\
&\leq&C\|\chi_{B}\|^{-1}_{L^{p(\cdot)}}\|(f-f_{B})\chi_{B}\|_{L^{p(\cdot)}}\\
&\leq&C\|f\|_{{\rm C}^{p(\cdot)}}.
\end{eqnarray*}
Therefore, we only need to prove that there exists a function $f$ such that $f \in {\rm CBMO}(\mathbb{R}^n)\backslash {\rm C}^{p(\cdot)}$. Without lose of generally, we may assume that $n=1$.

Let $A_{k}=\{x\in \mathbb{R}: 2^{k}<|x|\leq 2^{k}+1\}, k\in \mathbb{Z}_{+}$. Taking $f(x)=\sum_{k=0}^{\infty}2^{k}\chi_{A_{k}}(x)\mathrm{sgn}(x)$, then for any $B:=B(0,r)$,
$$f_{B}=\frac{1}{|B|}\int_{B}f(x)dx=0.$$
When $r\leq 1$, we have $f(x)\equiv 0$ and
\begin{equation}\label{1-1}
\sup_{0<r\leq 1}\frac{1}{|B|}\int_{B}|f(x)-f_{B}|dx=0.
\end{equation}
When $r>1$, letting $k_{0}\in \mathbb{Z}_{+}$ such that $2^{k_{0}}<r\leq 2^{k_{0}+1}$, then
\begin{equation}\label{1-2}
\sup_{r>1}\frac{1}{|B|}\int_{B}|f(x)-f_{B}|dx\leq \sup_{r>1}C2^{-k_{0}}\sum^{k_{0}+1}_{k=0}\int_{A_{k}}2^{k}dx\leq C.
\end{equation}
From (\ref{1-1}) and (\ref{1-2}), it follows that $f\in {\rm CBMO}(\mathbb{R}^n)$.

When $r>4$, letting $k_{0}\in \mathbb{Z}_{+}$ such that $2^{k_{0}}<r\leq 2^{k_{0}+1}$, thus $k_{0}\geq 2$ and
$$|(f-f_{B})\chi_{B}(x)|\geq \sum_{k=0}^{k_{0}-1}2^{k}\chi_{A_{k}}(x)\geq 2^{k_{0}-1}\chi_{A_{k_{0}-1}}(x)\geq Cr\chi_{A_{k_{0}-1}}(x),$$
which implies that
\begin{eqnarray*}
\sup_{r>0}\frac{\|(f-f_{B})\chi_{B}\|_{L^{p(\cdot)}(\mathbb{R}^n)}}{\|\chi_{B}\|_{L^{p(\cdot)}(\mathbb{R}^n)}}
&\geq& \sup_{r>4}\frac{\|(f-f_{B})\chi_{B}\|_{L^{p(\cdot)}(\mathbb{R}^n)}}{\|\chi_{B}\|_{L^{p(\cdot)}(\mathbb{R}^n)}}\\
&\geq& C\sup_{r>4}r\frac{\|\chi_{A_{k_{0}-1}}\|_{L^{p(\cdot)}(\mathbb{R}^n)}}{\|\chi_{B}\|_{L^{p(\cdot)}(\mathbb{R}^n)}}.
\end{eqnarray*}
Since $A_{k_{0}-1}\subset B$, by Lemma \ref{l5}, we have
$$\sup_{r>0}\frac{\|(f-f_{B})\chi_{B}\|_{L^{p(\cdot)}(\mathbb{R}^n)}}{\|\chi_{B}\|_{L^{p(\cdot)}(\mathbb{R}^n)}}\geq \sup_{r>4}Cr\bigg(\frac{|B|}{|A_{k_{0}-1}|}\bigg)^{-1/p_{0}}=\sup_{r>4}Cr^{1-1/p_{0}}=\infty.$$
Therefore, $f\not\in {\rm C}^{p(\cdot)}$.
\end{proof}

\begin{proposition}\label{p2}
If $p(\cdot)\in \mathcal{B}(\mathbb{R}^n)$, there is a constant $q>1$ such that ${\rm CBMO}^{q}(\mathbb{R}^n)\subset{\rm C}^{p(\cdot)}$.
\end{proposition}
\begin{proof}
We can take a cube $Q_{B}$ so that $B\subset Q_{B}\subset \sqrt{n}B$. By Lemma \ref{l2}, there exists a constant $0<\delta<1$ independent of $B$ such that for all $f\in L^{1}_{\rm loc}(\mathbb{R}^n)$,
$$\Big\||f|^{\delta}\chi_{Q_{B}}\Big\|_{L^{p(\cdot)}}\leq C(|f|_{Q_{B}})^{\delta}\|\chi_{Q_{B}}\|_{L^{p(\cdot)}}.$$
Now we put $q=1/\delta$ and $f=(b-b_{B})^{q}\chi_{B}$, we have
$$(|f|_{Q_{B}})^{\delta}=\bigg(\frac{1}{|Q_{B}|}\int_{B}|b(x)-b_{B}|^{q}dx\bigg)^{1/q}\leq C\|b\|_{{\rm CBMO}^{q}(\mathbb{R}^n)}.$$
By Lemma \ref{l5}, there exists a constant $1<p_{0}<p_{-}$ such that
\begin{eqnarray*}
\|\chi_{Q_{B}}\|_{L^{p(\cdot)}(\mathbb{R}^n)}
&\leq&\frac{\|\chi_{\sqrt{n}B}\|_{L^{p(\cdot)}(\mathbb{R}^n)}}{\|\chi_{B}\|_{L^{p(\cdot)}(\mathbb{R}^n)}}\|\chi_{B}\|_{L^{p(\cdot)}(\mathbb{R}^n)}\\
&\leq&C\bigg(\frac{|\sqrt{n}B|}{|B|}\bigg)^{1/p_{0}}\|\chi_{B}\|_{L^{p(\cdot)}(\mathbb{R}^n)}\\
&\leq&C\|\chi_{B}\|_{L^{p(\cdot)}(\mathbb{R}^n)}
\end{eqnarray*}
This gives us that
$$\|b\|_{{\rm C}^{p(\cdot)}}\leq C\|b\|_{{\rm CBMO}^{q}(\mathbb{R}^n)}.$$
Hence the proof of Proposition \ref{p2} is completed.
\end{proof}
\vspace{0.3cm}

\begin{proposition}\label{p3}
If $p(\cdot)\in \mathcal{B}(\mathbb{R}^n)$, then $f\in {\rm C}^{p(\cdot)}$ if and only if there exists a collection of numbers $\{c_{B}\}_{B}$ (i.e. for each ball $B$, there exists $c_{B}\in \mathbb{R}$) such that
$$\|f\|_{{\rm C}_{*}^{p(\cdot)}}
:=\sup_{r>0}\|\chi_{B(0,r)}\|^{-1}_{L^{p(\cdot)}(\mathbb{R}^n)}\|(f-c_{B(0,r)})\chi_{B(0,r)}\|_{L^{p(\cdot)}(\mathbb{R}^n)}<\infty.$$
\end{proposition}
\begin{proof}
We set $c_{B}=f_{B}$ for all balls $B$, the inequality $\|\cdot\|_{{\rm C}_{*}^{p(\cdot)}}\leq \|\cdot\|_{{\rm C}^{p(\cdot)}}$ holds. Let us check that the reverse inequality.

A similar argument as Proposition \ref{p1}, we have for any $B:=B(0,r)$,
\begin{eqnarray*}
\frac{1}{|B|}\int_{B}|f(x)-c_{B}|dx\leq C\|f\|_{{\rm C}_{*}^{p(\cdot)}}.
\end{eqnarray*}
Thus,
\begin{eqnarray*}
\frac{\|(f-f_{B})\chi_{B}\|_{L^{p(\cdot)}(\mathbb{R}^n)}}{\|\chi_{B}\|_{L^{p(\cdot)}(\mathbb{R}^n)}}
&\leq& \frac{\|(f-c_{B})\chi_{B}\|_{L^{p(\cdot)}(\mathbb{R}^n)}}{\|\chi_{B}\|_{L^{p(\cdot)}(\mathbb{R}^n)}}
+\frac{\|(c_{B}-f_{B})\chi_{B}\|_{L^{p(\cdot)}(\mathbb{R}^n)}}{\|\chi_{B}\|_{L^{p(\cdot)}(\mathbb{R}^n)}}\\
&\leq& \|f\|_{{\rm C}_{*}^{p(\cdot)}}+|c_{B}-f_{B}|\\
&\leq& C\|f\|_{{\rm C}_{*}^{p(\cdot)}}.
\end{eqnarray*}
Therefore, the proof of Proposition \ref{p3} is completed.
\end{proof}
\begin{proposition}\label{p4}
If $p(\cdot)\in \mathcal{B}(\mathbb{R}^n)$, then $f\in {\rm C}^{p(\cdot)}$ if and only if
$$\|f\|_{{\rm C}_{**}^{p(\cdot)}}
:=\sup_{r>0}\inf_{c}\|\chi_{B(0,r)}\|^{-1}_{L^{p(\cdot)}(\mathbb{R}^n)}\|(f-c)\chi_{B(0,r)}\|_{L^{p(\cdot)}(\mathbb{R}^n)}<\infty.$$
\end{proposition}
\begin{proof}
The proof of Proposition \ref{p4} is similar as that of Proposition \ref{p3}, we omit the details.
\end{proof}
\section{Characterization of ${\rm C}^{p(\cdot)}$ spaces via commutators}

We first review the definitions of Hardy operator and its dual operator. For a locally integrable function $f$ in $\mathbb{R}^n$, the $n$-dimensional Hardy operator $H$ is defined by
$$Hf(x)=\frac{1}{|x|^{n}}\int_{|y|\leq |x|}f(y)dy,\qquad x\in \mathbb{R}^n\setminus \{0\}.$$
The dual operator of $H$ is $H^{*}$ which is defined by
$$H^{*}f(x)=\int_{|y|>|x|}\frac{f(y)}{|y|^{n}}dy, \qquad x\in \mathbb{R}^n\setminus \{0\}.$$

The study of Hardy operator has a very long history and number of papers involved its generalizations, variants and applications. For the earlier development of this kind of integrals and many important applications, we refer the interested reader to the book \cite{HLP}. We are interested in the characterization of the commutator of Hardy operator. Now we give a remarkable result about the commutator of fractional Hardy operator; that is, Fu, Liu, Lu and Wang \cite{FLLW} showed that

\textbf{Theorem A.}
Let $\beta\geq 0, 1<p_{1}<\infty, 1/p_{1}-1/p_{2}=\beta/n, b\in {\rm CBMO}^{\max(p_{2},p_{1}')}(\mathbb{R}^n)$. Then both $[b,H_{\beta}]$ and $[b,H^{*}_{\beta}]$ are bounded from $L^{p_{1}}(\mathbb{R}^n)$ to $L^{p_{2}}(\mathbb{R}^n)$. Conversely,
\begin{enumerate}
\item [\rm(1)] if $[b,H_{\beta}]$ is bounded from $L^{p_{1}}(\mathbb{R}^n)$ to $L^{p_{2}}(\mathbb{R}^n)$, then $b\in {\rm CBMO}^{p'_{1}}(\mathbb{R}^n)$;
\item [\rm(2)] if $[b,H^{*}_{\beta}]$ is bounded from $L^{p_{1}}(\mathbb{R}^n)$ to $L^{p_{2}}(\mathbb{R}^n)$, then $b\in {\rm CBMO}^{p_{2}}(\mathbb{R}^n)$.
\end{enumerate}

Our main result as follows.
\begin{theorem}\label{main1}
If $p(\cdot)\in \mathcal{B}(\mathbb{R}^n)$, the following statements are equivalent:
\begin{enumerate}
\item [\rm(1)] $b\in {\rm C}^{p(\cdot)}\bigcap {\rm C}^{p'(\cdot)}.$
\item [\rm(2)] $[b,H]$ and $[b,H^{*}]$ are bounded from $L^{p(\cdot)}(\mathbb{R}^n)$ to $L^{p(\cdot)}(\mathbb{R}^n)$.
\end{enumerate}
\end{theorem}

\begin{proof}
$(1)\Rightarrow (2)$. We focus on the proof of the boundedness of $[b,H]$, since the arguments of $[b,H^{*}]$ are similar with necessary modifications.

For any $f\in L^{p(\cdot)}(\mathbb{R}^n)$, we have
\begin{eqnarray*}
\|[b,H](f)\|_{L^{p(\cdot)}}&=&\sum_{k=-\infty}^{\infty}\|\chi_{k}[b,H](f)\|_{L^{p(\cdot)}(\mathbb{R}^n)}\\
&=&\sum_{k=-\infty}^{\infty}\Big\|\frac{\chi_{k}(\cdot)}{|\cdot|^{n}}\int_{B(0,|\cdot|)}\big(b(\cdot)-b(y)\big)f(y)dy\Big\|_{L^{p(\cdot)}(\mathbb{R}^n)}\\
&\leq&\sum_{k=-\infty}^{\infty}\Big\|\chi_{k}(\cdot)
\sum_{j=-\infty}^{k}\frac{1}{|\cdot|^{n}}\int_{C_{j}}\big|b(\cdot)-b(y)\big||f(y)|dy\Big\|_{L^{p(\cdot)}(\mathbb{R}^n)}.
\end{eqnarray*}
It is easy to see that
\begin{eqnarray*}
\int_{C_{j}}\big|b(x)-b(y)\big||f(y)|dy&\leq&\int_{C_{j}}\big|b(x)-b_{B_{k}}\big||f(y)|dy+\int_{C_{j}}\big|b_{B_{k}}-b(y)\big||f(y)|dy
\end{eqnarray*}
By \eqref{1.1}, we get
\begin{eqnarray}\label{4.1}
\begin{split}
\int_{C_{j}}\big|b(x)-b_{B_{k}}\big||f(y)|dy
&\leq C\big|b(x)-b_{B_{k}}\big|\|\chi_{j}\|_{L^{p'(\cdot)}(\mathbb{R}^n)}\|f\chi_{j}\|_{L^{p(\cdot)}(\mathbb{R}^n)}
\end{split}
\end{eqnarray}
In \cite{FLLW}, Fu, Liu, Lu and Wang showed that for $b\in {\rm CBMO}(\mathbb{R}^n)$ and $j,k\in \mathbb{Z}$,
$$|b(t)-b_{B_{k}}|\leq |b(t)-b_{B_{j}}|+C|j-k|\|b\|_{{\rm CBMO}(\mathbb{R}^n)}.$$
By Proposition \ref{p1}, for $b\in {\rm C}^{p'(\cdot)}\subset {\rm CBMO}(\mathbb{R}^n)$ and $j,k\in \mathbb{Z}$, we have
$$|b(t)-b_{B_{k}}|\leq |b(t)-b_{B_{j}}|+C|j-k|\|b\|_{{\rm C}^{p'(\cdot)}}.$$
This gives us that
\begin{eqnarray}\label{4.2}
\begin{split}
&\int_{C_{j}}\big|b(y)-b_{B_{k}}\big||f(y)|dy\\
&\leq \int_{C_{j}}\big|b(y)-b_{B_{j}}\big||f(y)|dy+\big|b_{B_{k}}-b_{B_{j}}\big|\int_{C_{j}}|f(y)|dy\\
&\leq C\|(b-b_{B_{j}})\chi_{j}\|_{L^{p'(\cdot)}(\mathbb{R}^n)}\|f\chi_{j}\|_{L^{p(\cdot)}(\mathbb{R}^n)}
+C(k-j)\|b\|_{{\rm C}^{p'(\cdot)}}\|\chi_{j}\|_{L^{p'(\cdot)}(\mathbb{R}^n)}\|f\chi_{j}\|_{L^{p(\cdot)}(\mathbb{R}^n)}\\
&\leq C\|b\|_{{\rm C}^{p'(\cdot)}}\|\chi_{j}\|_{L^{p'(\cdot)}(\mathbb{R}^n)}\|f\chi_{j}\|_{L^{p(\cdot)}(\mathbb{R}^n)}
+C(k-j)\|b\|_{{\rm C}^{p'(\cdot)}}\|\chi_{j}\|_{L^{p'(\cdot)}(\mathbb{R}^n)}\|f\chi_{j}\|_{L^{p(\cdot)}(\mathbb{R}^n)}\\
&\leq C(k-j)\|b\|_{{\rm C}^{p'(\cdot)}}\|\chi_{j}\|_{L^{p'(\cdot)}(\mathbb{R}^n)}\|f\chi_{j}\|_{L^{p(\cdot)}(\mathbb{R}^n)}.
\end{split}
\end{eqnarray}
Combing \eqref{4.1} and \eqref{4.2}, we get
\begin{eqnarray*}
&&\Big\|\chi_{k}(\cdot)\frac{1}{|\cdot|^{n}}\int_{C_{j}}\big(b(\cdot)-b(y)\big)f(y)dy\Big\|_{L^{p(\cdot)}(\mathbb{R}^n)}\\
&&\leq C2^{-kn}\Big\|\big(b-b_{B_{k}}\big)\chi_{k}\Big\|_{L^{p(\cdot)}(\mathbb{R}^n)}\|\chi_{j}\|_{L^{p'(\cdot)}(\mathbb{R}^n)}\|f\chi_{j}\|_{L^{p(\cdot)}(\mathbb{R}^n)}\\
&&\quad+C2^{-kn}(k-j)\|\chi_{k}\|_{L^{p(\cdot)}(\mathbb{R}^n)}\|\chi_{j}\|_{L^{p'(\cdot)}(\mathbb{R}^n)}\|f\chi_{j}\|_{L^{p(\cdot)}(\mathbb{R}^n)}\\
&&\leq C2^{-kn}(k-j)\|\chi_{k}\|_{L^{p(\cdot)}(\mathbb{R}^n)}\|\chi_{j}\|_{L^{p'(\cdot)}(\mathbb{R}^n)}\|f\chi_{j}\|_{L^{p(\cdot)}(\mathbb{R}^n)}.
\end{eqnarray*}
From lemma \ref{l3} and \ref{l4}, it follows that for $k\geq j$, there exists a constant $\delta>0$ such that
\begin{eqnarray*}
2^{-kn}\|\chi_{k}\|_{L^{p(\cdot)}(\mathbb{R}^n)}\|\chi_{j}\|_{L^{p'(\cdot)}(\mathbb{R}^n)}&\leq& C\frac{\|\chi_{j}\|_{L^{p'(\cdot)}(\mathbb{R}^n)}}{\|\chi_{k}\|_{L^{p'(\cdot)}(\mathbb{R}^n)}}\\
&\leq&C2^{n\delta(j-k)}.
\end{eqnarray*}
Therefore, generalized Minkowski inequality imply
\begin{eqnarray*}
\|[b,H](f)\|_{L^{p(\cdot)}(\mathbb{R}^n)}&\leq&C\sum_{k=-\infty}^{\infty}\sum_{j=-\infty}^{k}(k-j)2^{n\delta(j-k)}\|f\chi_{j}\|_{L^{p(\cdot)}(\mathbb{R}^n)}\\
&\leq&C\sum_{j=-\infty}^{\infty}\sum_{k=j}^{\infty}(k-j)2^{n\delta(j-k)}\|f\chi_{j}\|_{L^{p(\cdot)}(\mathbb{R}^n)}\\
&\leq&C\|f\|_{L^{p(\cdot)}(\mathbb{R}^n)}.
\end{eqnarray*}

$(2)\Rightarrow (1)$. The condition $b\in {\rm C}^{p}\bigcap {\rm C}^{p'}$ turns out to be necessary for the conclusion that both $[b,H]$ and $[b,H^{*}]$ are bounded from $L^{p(\cdot)}(\mathbb{R}^n)$ to $L^{p(\cdot)}(\mathbb{R}^n)$ holds.

For any ball $B:=B(0,r)$ and $x\in B$, we have
\begin{eqnarray*}
&&|b(x)-b_{B}|=\Big|\frac{1}{|B|}\int_{B}\big(b(x)-b(y)\big)dy\Big|\\
&&\leq C\Big|\frac{1}{|x|^{n}}\int_{|y|\leq |x|}\big(b(x)-b(y)\big)\chi_{B}(y)dy\Big|
+C\Big|\int_{|x|\leq |y|}\frac{\big(b(x)-b(y)\big)\chi_{B}(y)|y|^{n}|B|^{-1}}{|y|^{n}}dy\Big|\\
&&\leq C|[b,H](\chi_{B})(x)|+C|[b,H^{*}](f_{0})(x)|,
\end{eqnarray*}
where $f_{0}(x)=|x|^{n}|B|^{-1}\chi_{B}(x)$.

From $[b,H]$ and $[b,H^{*}]$ are bounded on $L^{p(\cdot)}(\mathbb{R}^n)$, it follows that
\begin{eqnarray*}
\|(b-b_{B})\chi_{B}\|_{L^{p(\cdot)}(\mathbb{R}^n)}
&\leq&C\|[b,H](\chi_{B})\|_{L^{p(\cdot)}(\mathbb{R}^n)}+\|[b,H^{*}](f_{0})\|_{L^{p(\cdot)}(\mathbb{R}^n)}\\
&\leq& C\|\chi_{B}\|_{L^{p(\cdot)}(\mathbb{R}^n)}+C\|f_{0}\|_{L^{p(\cdot)}(\mathbb{R}^n)}\\
&\leq& C\|\chi_{B}\|_{L^{p(\cdot)}(\mathbb{R}^n)}.
\end{eqnarray*}
Therefore, we obtain that $b$ belongs to ${\rm C}^{p(\cdot)}$.

By \eqref{1.1}, we know that both $[b,H]$ and $[b,H^{*}]$ are bounded on $L^{p'(\cdot)}(\mathbb{R}^n)$. Therefore, we also obtain that $b\in {\rm C}^{p'(\cdot)}$.
\end{proof}

\section{Vector-valued inequality}

Now we give the definition of Herz spaces with variable exponent (\cite{I}). Let $\alpha\in \mathbb{R}, 0<q<\infty$ and $p(\cdot)\in \mathcal{P}(\mathbb{R}^n)$. The homogeneous Herz space $\dot{K}^{\alpha,q}_{p(\cdot)}(\mathbb{R}^n)$ is defined by
$$\dot{K}^{\alpha,q}_{p(\cdot)}(\mathbb{R}^n):=\Big\{f\in L^{p(\cdot)}_{\rm loc}(\mathbb{R}^n\backslash \{0\}):\|f\|_{\dot{K}^{\alpha,q}_{p(\cdot)}(\mathbb{R}^n)}<\infty\Big\},$$
where
$$\|f\|_{\dot{K}^{\alpha,q}_{p(\cdot)}(\mathbb{R}^n)}:=\Big\|\Big\{2^{\alpha k}\|f\chi_{k}\|_{L^{p(\cdot)}(\mathbb{R}^n)}\Big\}_{k=-\infty}^{\infty}\Big\|_{\ell^{q}}.$$

We prove the boundedness of vector-valued commutator of Hardy operator on Herz spaces with variable exponent.

\begin{theorem}\label{main2}
Let $p(\cdot)\in \mathcal{B}(\mathbb{R}^n)$, $1<r<\infty$, $0<q<\infty$, $\alpha<n/p'_{-}$ and $b\in C^{p(\cdot)}\bigcap C^{p'(\cdot)}$. Then there is a constant $C$ such that
\begin{equation*}
\bigg\|\Big(\sum_{j=1}^{\infty}|[b,H](f_{j})|^{r}\Big)^{1/r}\bigg\|_{\dot{K}^{\alpha,q}_{p(\cdot)}(\mathbb{R}^n)}\leq C\bigg\|\Big(\sum_{j=1}^{\infty}|f_{j}|^{r}\Big)^{1/r}\bigg\|_{\dot{K}^{\alpha,q}_{p(\cdot)}(\mathbb{R}^n)}
\end{equation*}
and
\begin{equation*}
\bigg\|\Big(\sum_{j=1}^{\infty}|[b,H^*](f_{j})|^{r}\Big)^{1/r}\bigg\|_{\dot{K}^{\alpha,q}_{p(\cdot)}(\mathbb{R}^n)}\leq C\bigg\|\Big(\sum_{j=1}^{\infty}|f_{j}|^{r}\Big)^{1/r}\bigg\|_{\dot{K}^{\alpha,q}_{p(\cdot)}(\mathbb{R}^n)}
\end{equation*}
for all sequences of functions $\{f_{j}\}_{j=1}^{\infty}$ satisfying $\|\|\{f_{j}\}_{j}\|_{\ell^{r}}\|_{\dot{K}^{\alpha,q}_{p(\cdot)}(\mathbb{R}^n)}<\infty$.
\end{theorem}

To order to prove Theorem \ref{main2}, we additionally introduce the next lemma well-known as the generalized Minkowski inequality.
\begin{lemma}\label{l7}
If $1<r<\infty$, then there exists a constant $C>0$ such that for all sequences of functions $\{f_{j}\}_{j=1}^{\infty}$ satisfying $\|\|\{f_{j}\}_{j}\|_{\ell^{r}}\|_{L^{1}(\mathbb{R}^n)}<\infty$,
\begin{equation}\label{5.1}
\bigg\{\sum_{j=1}^{\infty}\bigg(\int_{\mathbb{R}^n}|f_{j}(y)|dy\bigg)^{r}\bigg\}^{1/r}\leq C\int_{\mathbb{R}^n}\bigg\{\sum_{j=1}^{\infty}|f_{j}(y)|^{r}\bigg\}^{1/r}dy.
\end{equation}
\end{lemma}

\vskip 0.5cm
\noindent
{\it Proof of Theorem \ref{main2}.}
For every $\{f_{j}\}_{j=1}^{\infty}$ with $\|\|\{f_{j}\}_{j}\|_{\ell^{r}}\|_{\dot{K}^{\alpha,q}_{p(\cdot)}(\mathbb{R}^n)}<\infty$, we obtain
\begin{eqnarray*}
\|\|\{[b,H](f_{j})\}_{j}\|_{\ell^{r}}\|_{\dot{K}^{\alpha,q}_{p(\cdot)}(\mathbb{R}^n)}
&=&\bigg\{\sum_{k=-\infty}^{\infty}2^{\alpha q k}\Big\|\chi_{k}\Big\|\Big\{[b,H]\Big(\sum_{l=-\infty}^{\infty}f_{j}\chi_{l}\Big)\Big\}_{j}\Big\|_{\ell^{r}}\Big\|_{L^{p(\cdot)}(\mathbb{R}^n)}^{q}\bigg\}^{1/q}\\
&\leq& \bigg\{\sum_{k=-\infty}^{\infty}2^{\alpha q k}\Big\|\chi_{k}\sum_{l=-\infty}^{k}\Big\|\Big\{[b,H]\Big(f_{j}\chi_{l}\Big)\Big\}_{j}\Big\|_{\ell^{r}}\Big\|_{L^{p(\cdot)}(\mathbb{R}^n)}^{q}\bigg\}^{1/q}.
\end{eqnarray*}
For convenience below we denote $F:=\|\{f_{j}\}_{j}\|_{\ell^{r}}$. For $x\in C_{k}$, generalized H\"{o}lder inequality and generalized Minkowski inequality \eqref{5.1} imply
\begin{eqnarray*}
&&\|\{[b,H](f_{j}\chi_{l})(x)\}_{j}\|_{\ell^{r}}\\
&&\leq C\bigg\|\bigg\{\frac{1}{|x|^n}\int_{C_{l}}|b(x)-b(y)||f_{j}(y)|dy\bigg\}_{j}\bigg\|_{\ell^{r}}\\
&&\leq C2^{-kn}\bigg\|\bigg\{\int_{C_{l}}|b(x)-b(y)||f_{j}(y)|dy\bigg\}_{j}\bigg\|_{\ell^{r}}\\
&&\leq C2^{-kn}\int_{C_{l}}|b(x)-b(y)|F(y)dy\\
&&\leq C2^{-kn}\bigg\{|b(x)-b_{B_{l}}|\int_{C_{l}}F(y)dy+\int_{C_{l}}|b_{B_{l}}-b(y)|F(y)dy\bigg\}\\
&&\leq C2^{-kn}\|F\chi_{l}\|_{L^{p(\cdot)}(\mathbb{R}^n)}\bigg\{|b(x)-b_{B_{l}}|\|\chi_{l}\|_{L^{p'(\cdot)}(\mathbb{R}^n)}
+\Big\||b-b_{B_{l}}|\chi_{l}\Big\|_{L^{p'(\cdot)}(\mathbb{R}^n)}\bigg\}.
\end{eqnarray*}
From the fact that
$$|b(x)-b_{B_{l}}|\leq |b(t)-b_{B_{k}}|+C|l-k|\|b\|_{{\rm C}^{p(\cdot)}}.$$
Which gives us that
\begin{eqnarray*}
&&\|\chi_{k}\|\{[b,H](f_{j}\chi_{l})\}_{j}\|_{\ell^{r}}\|_{L^{p(\cdot)}(\mathbb{R}^n)}\\
&&\leq C2^{-kn}\|F\chi_{l}\|_{L^{p(\cdot)}(\mathbb{R}^n)}\\
&&\quad\times\Big\{\|(b-b_{B_{l}})\chi_{k}\|_{{L}^{p(\cdot)}(\mathbb{R}^n)}\|\chi_{l}\|_{L^{p'(\cdot)}(\mathbb{R}^n)}
+\|(b_{B_{l}}-b)\chi_{l}\|_{L^{p'(\cdot)}(\mathbb{R}^n)}\|\chi_{k}\|_{L^{p(\cdot)}(\mathbb{R}^n)}\Big\}\\
&&\leq C2^{-kn}\|F\chi_{l}\|_{L^{p(\cdot)}(\mathbb{R}^n)}\\
&&\quad \times\Big\{(k-l)\|b\|_{{\rm C}^{p(\cdot)}}\|\chi_{l}\|_{L^{p'(\cdot)}(\mathbb{R}^n)}\|\chi_{k}\|_{L^{p(\cdot)}(\mathbb{R}^n)}
+\|b\|_{{\rm C}^{p'(\cdot)}(\mathbb{R}^n)}\|\chi_{k}\|_{L^{p(\cdot)}(\mathbb{R}^n)}\|\chi_{l}\|_{L^{p'(\cdot)}(\mathbb{R}^{n})}\Big\}\\
&&\leq C(k-l)\|F\chi_{l}\|_{L^{p(\cdot)}(\mathbb{R}^n)}2^{-kn} \|\chi_{k}\|_{L^{p(\cdot)}(\mathbb{R}^n)}\|\chi_{l}\|_{L^{p'(\cdot)}(\mathbb{R}^{n})}\\
&&\leq C(k-l)\|F\chi_{l}\|_{L^{p(\cdot)}(\mathbb{R}^n)}2^{-kn} \|\chi_{k}\|_{L^{p(\cdot)}(\mathbb{R}^n)}\Big(|B_{l}|\|\chi_{l}\|^{-1}_{L^{p(\cdot)}(\mathbb{R}^n)}\Big).
\end{eqnarray*}
By Lemma \ref{l5} and $\alpha/n<1/p'_{-}$, there exists a constant $1<p_{0}<p_{-}$ such that $\alpha/n<1/p'_{0}$ and
$$\frac{\|\chi_{k}\|_{L^{p(\cdot)}(\mathbb{R}^n)}}{\|\chi_{l}\|_{L^{p(\cdot)}(\mathbb{R}^n)}}\leq C2^{n(k-l)/p_{0}}.$$
Then
$$\|\chi_{k}\|\{[b,H](f_{j}\chi_{l})\}_{j}\|_{\ell^{r}}\|_{L^{p(\cdot)}(\mathbb{R}^n)}\leq C(k-l)2^{n(l-k)/p'_{0}}\|F\chi_{l}\|_{L^{p(\cdot)}(\mathbb{R}^n)}.$$
This implies that
\begin{eqnarray*}
\|\|\{[b,H](f_{j})\}_{j}\|_{\ell^{r}}\|_{\dot{K}^{\alpha,q}_{p(\cdot)}(\mathbb{R}^n)}
&\leq& C\bigg\{\sum_{k=-\infty}^{\infty}2^{\alpha qk}\bigg(\sum_{l=-\infty}^{k}(k-l)2^{n(l-k)/p'_{0}}\|F\chi_{l}\|_{L^{p(\cdot)}(\mathbb{R}^n)}\bigg)^{q}\bigg\}^{1/q}\\
&=& C\bigg\{\sum_{k=-\infty}^{\infty}\bigg(\sum_{l=-\infty}^{k}2^{\alpha l}(k-l)2^{(n/p'_{0}-\alpha)(l-k)}\|F\chi_{l}\|_{L^{p(\cdot)}(\mathbb{R}^n)}\bigg)^{q}\bigg\}^{1/q}.
\end{eqnarray*}
We consider the two cases "$1<q<\infty$" and "$0<q\leq 1$".

If $1<q<\infty$, then we use H\"{o}lder's inequality and obtain
\begin{eqnarray*}
\|\|\{[b,H](f_{j})\}_{j}\|_{\ell^{r}}\|_{\dot{K}^{\alpha,q}_{p(\cdot)}(\mathbb{R}^n)}
&\leq& C\bigg\{\sum_{k=-\infty}^{\infty}\bigg(\sum^{k}_{l=-\infty}2^{\alpha ql}\|F\chi_{l}\|_{L^{p(\cdot)}(\mathbb{R}^n)}^{q}2^{(n/p'_{0}-\alpha)q(l-k)/2}\bigg)\\
&&\times\bigg(\sum_{l=-\infty}^{k}2^{(n/p'_{0}-\alpha)q'(l-k)/2}(k-l)^{q'}\bigg)^{q/q'}\bigg\}^{1/q}\\
&\leq& C\bigg(\sum_{k=-\infty}^{\infty}\sum_{l=-\infty}^{k}2^{\alpha ql}\|F\chi_{l}\|_{L^{p(\cdot)}(\mathbb{R}^n)}^{q}2^{(n/p'_{0}-\alpha)q(l-k)/2}\bigg)^{1/q}\\
&\leq& C\bigg(\sum_{l=-\infty}^{\infty}2^{\alpha ql}\|F\chi_{l}\|_{L^{p(\cdot)}(\mathbb{R}^n)}^{q}\sum_{k=l}^{\infty}2^{(n/p'_{0}-\alpha)q(l-k)/2}\bigg)^{1/q}\\
&\leq& C\bigg(\sum_{l=-\infty}^{\infty}2^{\alpha ql}\|F\chi_{l}\|_{L^{p(\cdot)}(\mathbb{R}^n)}^{q}\bigg)^{1/q}\\
&\leq& C\|F\|_{\dot{K}^{\alpha,q}_{p(\cdot)}(\mathbb{R}^n)}.
\end{eqnarray*}

If $0<q\leq 1$, then we use the following inequality
$$\Big(\sum_{i=1}^{\infty}|a_{i}|\Big)^{q}\leq \sum_{i=1}^{\infty}|a_{i}|^{q},$$
and obtain
\begin{eqnarray*}
\|\|\{[b,H](f_{j})\}_{j}\|_{\ell^{r}}\|_{\dot{K}^{\alpha,q}_{p(\cdot)}(\mathbb{R}^n)}
&\leq& C\bigg\{\sum_{k=-\infty}^{\infty}\sum^{k}_{l=-\infty}2^{\alpha ql}\|F\chi_{l}\|_{L^{p(\cdot)}(\mathbb{R}^n)}^{q}2^{(n/p'_{0}-\alpha)q(l-k)}(k-l)^{q}\bigg\}^{1/q}\\
&\leq& C\bigg(\sum_{l=-\infty}^{\infty}2^{\alpha ql}\|F\chi_{l}\|_{L^{p(\cdot)}(\mathbb{R}^n)}^{q}\sum_{k=l}^{\infty}2^{(n/p'_{0}-\alpha)q(l-k)}(k-l)^{q}\bigg)^{1/q}\\
&\leq& C\bigg(\sum_{l=-\infty}^{\infty}2^{\alpha ql}\|F\chi_{l}\|_{L^{p(\cdot)}(\mathbb{R}^n)}^{q}\bigg)^{1/q}\\
&\leq& C\|F\|_{\dot{K}^{\alpha,q}_{p(\cdot)}(\mathbb{R}^n)}.
\end{eqnarray*}
This completes the proof of Theorem \ref{main2}. \qed

\vspace{0.3cm}

\acknowledgements{\rm We thank the referees for their time and
comments. }

\end{document}